\documentclass[12pt]{amsart}

\usepackage{amsfonts,amssymb}
\usepackage{hyperref, graphicx}
\usepackage{epsfig}
\usepackage{latexsym}
\usepackage{stmaryrd}
\usepackage{amsmath,amsthm}
\usepackage{mathrsfs}
\usepackage[all,cmtip]{xy}
\usepackage{tikz-cd}
\theoremstyle{plain}
\newtheorem{theorem}{Theorem}[section]
\newtheorem{lemma}[theorem]{Lemma}
\newtheorem{definition}[theorem]{Definition}
\newtheorem{proposition}[theorem]{Proposition}
\newtheorem{cor}[theorem]{Corollary}
\newtheorem{remark}[theorem]{Remark}
\newtheorem{example}[theorem]{Example}

\newtheorem{convention}{Convention}

\numberwithin{equation}{section}

\newcommand{\ra}{\rightarrow}
\newcommand{\Z}{\mathbb{Z}}

\newcommand{\U}{\mathbf{U}}
\newcommand{\sheaf}[1]{\mathcal{#1}}
\newcommand{\id}{\mathbf{1}}

\newcommand{\Hom}[3][]{\mathrm{Hom}_{#1}(#2, #3)}

\newcommand{\Sch}{\mathbf{Sch}}
\newcommand{\Set}{\mathbf{Set}}
\newcommand{\M}{\mathfrak{M}}
\newcommand{\Ms}{\mathcal{M}}
\newcommand{\dv}{\textbf{d}}
\newcommand{\QAl}{\ensuremath{{\mathbf{A_\ell} \times Q}}}
\newcommand{\catQAl}{\ensuremath{\mathcal{M}(\QAl)}}
\newcommand{\catchain}{\ensuremath{\mathcal{M}_{\mathrm{rel}}(\QAl)}}
\newcommand{\catfilt}{\ensuremath{\mathcal{M}_{\mathrm{fil}}(\QAl)}}
\newcommand{\mf}{\mathfrak{M}^\mathrm{ss}}
\newcommand{\mfc}{\mathfrak{M}_{\mathrm{rel}}^\mathrm{ss}}
\newcommand{\mff}{\mathfrak{M}_{\mathrm{fil}}^\mathrm{ss}}
\newcommand{\mtf}{\mathfrak{M}^\mathrm{s}}
\newcommand{\mtfc}{\mathfrak{M}_{\mathrm{rel}}^\mathrm{s}}
\newcommand{\mtff}{\mathfrak{M}_{\mathrm{fil}}^\mathrm{s}}
\newcommand{\ms}{M^\mathrm{ss}}
\newcommand{\msc}{M^\mathrm{ss}_{\mathrm{rel}}}
\newcommand{\msf}{M^\mathrm{ss}_{\mathrm{fil}}}
\newcommand{\mst}{M^\mathrm{s}}
\newcommand{\mstc}{M^\mathrm{s}_{\mathrm{rel}}}
\newcommand{\mstf}{M^\mathrm{s}_{\mathrm{fil}}}
\newcommand{\rs}{\mathcal{R}}
\newcommand{\rsc}{\mathcal{R}_{\mathrm{rel}}}
\newcommand{\rsf}{\mathcal{R}_{\mathrm{fil}}}
\newcommand{\fp}[1]{\underline{#1}}

\newcommand{\dimv}[1]{\ensuremath{\mathbf{#1}}}

\hypersetup{colorlinks, citecolor=blue, filecolor=black, linkcolor=red}

\begin{document}
\baselineskip=15.5pt

\title{Moduli of filtered quiver representations}
 \author{Sanjay~Amrutiya}
 \address{Department of Mathematics, IIT Gandhinagar,
 Near Village Palaj, Gandhinagar - 382355, India}
 \email{samrutiya@iitgn.ac.in}
 \author{Umesh~V.~Dubey}
 \address{ Harish-Chandra Research Institute, HBNI,
Chhatnag Road, Jhunsi, Allahabad - 211019, India}	
 \email{umeshdubey@hri.res.in}
\date{}

\begin{abstract}
\noindent In this paper, we give a construction of the moduli space of filtered 
representations of a given quiver of fixed dimension vector with the appropriate 
notion of stability. The construction of the moduli of filtered representations uses 
the moduli of representations of ladder quiver. The ladder quiver is introduced using 
a given quiver and an $A_n$-type quiver. We also study determinantal
theta functions on such moduli spaces.

\end{abstract}
\maketitle
\section{Introduction}
In \cite{Ki94}, A. King has used geometric invariant theory to construct moduli of 
quiver representations.  Since the category of filtered representations may not be an 
abelian category, we cannot apply the methods of A.~King \cite{Ki94}  directly for 
the construction. 
Once we suitably choose the stability parameter, we get a quasi-abelian category of 
semistable representations having a fixed slope.
We generalize the approach of \cite{Ki94} in the filtered representation category 
setup via the concept of slope stability studied by Andr\'e \cite{An09} for 
quasi-abelian categories.

The moduli construction in this note will be used in the forthcoming work on the 
functorial moduli construction of parabolic sheaves. Besides this, it is 
interesting to investigate some geometric properties of moduli of filtered quiver 
representations, in general.

We first review some basic properties of filtered representations of a quiver $Q$
in Section \ref{sec-2}. We introduce the ladder quiver $\QAl$ (see Section 
\ref{ladder_quiver}), an admissible ideal in the path algebra of a ladder quiver. 
We describe the category of filtered representations of a quiver $Q$ in terms of 
representations of the ladder quiver with relations.

We study the slope functions introduced by Andr\'e \cite{An09} in Section \ref{sec-rank-fn}.
Using the concept of a maximal flag, we generalize the definition of $S$-equivalence 
for filtered representations. We also give some examples of slope functions.

In Section \ref{sec-git-mc}, we construct the moduli space of filtered quiver 
representations (see Theorem \ref{main-theorem}). We prove that there is a canonical 
morphism from this moduli space to the moduli space of representations of a ladder 
quiver. For an appropriate choice of slope function, we show that this canonical 
morphism is an open immersion (see Corollary \ref{cor-open-immersion}).

In the final Section \ref{sec-det-fns}, we discuss a characterisation of semistability of filtered 
representations in terms of determinantal theta functions using the results of 
\cite{DW2}. We also study the corresponding morphism on such moduli spaces into a 
projective space.

\noindent \textbf{Notations:}
Let $\mathbb{K}$ be an algebraically closed field.
\begin{itemize}
\item $A$\text{-mod} = the category of finitely generated (left) modules over an algebra $A$.
\item $\catQAl$ = the category of ($\mathbb{K}$-linear) representation of $\QAl$.
\item $\catchain$ = the category of ($\mathbb{K}$-linear) representation of $\QAl/I$,  
i.e., the category of 
representation with the admissible ideal of relations $I$ defined in \ref{ladder_quiver}.
\item $\catfilt$ = the category of ($\mathbb{K}$-linear) filtered representation of $Q$ 
of length at most $\ell$.
\item $\Sch$ = the category of all $\mathbb{K}$-schemes.
\end{itemize}

\section{Filtered quiver representations}\label{sec-2}

Let $Q = (Q_0, Q_1, s, t)$ be a (finite) quiver without oriented cycles, and let 
$\mathbb{K}Q$ be its path algebra which is a finite dimensional $k$-algebra. 
We will some time ignore the maps $s$ and $t$ when it is clear from the context, 
for example, in the case of linear quiver $\mathbf A_\ell$ the vertex set is 
$\{ 1, \ldots, \ell \}$ with the set of arrows $\{ 1, \ldots, \ell-1 \}$ and the 
map $s$ (respectively, $t$) is the canonical inclusion (respectively, the shift 
$j \mapsto j +1$). The abelian category $\mathbb{K}Q$-mod is equivalent to the 
category of $\mathbb{K}$-linear representations of $Q$, see \cite{Re08} for more 
details. We will also use $M(v)\,, v \in Q_0,$ (respectively, $M(a); a \in Q_1$) 
to denote a vector space (respectively, a linear map) of a representation 
corresponding to the $\mathbb{K}Q$-module $M$ at the vertex $v$ (respectively, 
the arrow $a$) under this equivalence.
 
\begin{definition} \label{ladder_quiver}\rm{
If $Q= (Q_0, Q_1, s, t)$ is a quiver and $\mathbf A_\ell = 
(\mathbf A_{\ell 0}, \mathbf A_{\ell 1})$ is a linear quiver with $\ell$ vertices. 
Then the ladder quiver is defined as 
$$
\QAl := ({\mathbf A_{\ell 0}} \times Q_0, {\mathbf A_{\ell 0 }} \times Q_1 
\sqcup {\mathbf A_{\ell 1}} \times Q_0, s_\ell , t_\ell ).
$$
To simplify the notation, let $\alpha_i^a := (i , a) \in 
{\mathbf A_{\ell 0}} \times Q_1$ for $i = 1, \ldots, \ell$ 
(or $i \in \mathbf A_{\ell 0}$)  and $\beta_j^v := (j , v) \in 
{\mathbf A_{\ell 1}} \times Q_0$ for $j=1, \ldots, (\ell-1)$ 
(or $j \in \mathbf A_{\ell 1}$) represents the arrows in the ladder quiver $\QAl$. 
Then the source map for ladder quiver is defined as 
$s_\ell  (\alpha_i^a) = (i, s(a)), s_\ell (\beta_j^v) = (j, v)$ and the target map 
is defined as $t_\ell (\alpha_i^a) = (i , t(a))$ and $ t_\ell  (\beta_j^v) = (j +1 , v) $.

Now using this notation for arrows, we can define the admissible ideal $I$ as the 
ideal generated by 
$ \beta_k^{t(a)} \alpha_k^a - \alpha_{k +1}^a \beta_k^{s(a)}$, where $a\in Q_1$ and 
$k = 1, \ldots, (\ell - 1)$ (or $k \in \mathbf A_{\ell 1}$).
}
\end{definition}

We shall denote by $\mathbb{K}(\QAl)$ the path-algebra of the quiver $\QAl$.

\begin{remark}\rm{
The definition of ladder quiver is a special case of more general notion of tensor 
product of two quivers defined in \cite[Section 3]{Ke13}
\footnote{We are grateful to B. Keller for pointing out this remark after the first 
version of this manuscript was put up on Arxiv.}. Moreover, the tensor product of 
path-algebras $\mathbb{K}\mathbf{A}_\ell\otimes_{\mathbb{K}} \mathbb{K}Q$ is isomorphic 
to the algebra $\mathbb{K}(\QAl)/I$. 
}
\end{remark}
\begin{convention} \label{sink_identification}\rm{
We will identify the vertex set $Q_0$ embedded in the vertex set of ladder quiver as 
$ \{\ell\} \times Q_0$, where the vertex $\ell$ is the sink of a linear quiver.
}
\end{convention}

We will now give some examples of ladder quivers of a given quiver.

\begin{example} \rm{
The quiver $\mathbf{A}_\ell \times \mathbf{A}_3$ can be described as follows:
$$
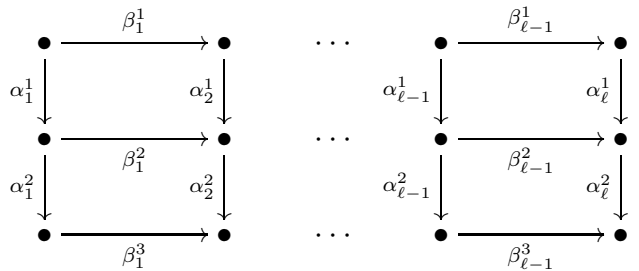
\xymatrix{
 \bullet \ar[d]_{\alpha_1^1} \ar[rr]^{\beta_1^1}  && \bullet \ar[d]_{\alpha_2^1} & \cdots & \bullet 
 \ar[d]_{\alpha_{\ell-1}^1} \ar[rr]^{\beta_{\ell-1}^1}  && \bullet \ar[d]_{\alpha_\ell^1} \\
 \bullet \ar[rr]_{\beta_1^2} \ar[d]_{\alpha_1^2} && \bullet \ar[d]_{\alpha_2^2} & \cdots &  \bullet 
 \ar[rr]_{\beta_{\ell-1}^2} \ar[d]_{\alpha_{\ell-1}^2} && \bullet \ar[d]_{\alpha_\ell^2} \\
  \bullet \ar[rr]_{\beta_1^3}  && \bullet  & \cdots  &  \bullet \ar[rr]_{\beta_{\ell-1}^3}  && \bullet 
}.
$$
}
\end{example}

\begin{example} \rm{
Let $Q = (Q_0, Q_1, s, t)$ be a quiver defined as follows:
The set vertices
$
Q_0 = \{ v_1, v_2, v_3, v_4 \}
$
and the set of arrows 
$
Q_1 = \{ a, b, c, d \}
$
with the source and target maps $s, t\colon Q_1 \ra Q_0$ given by 
\[
\begin{array}{cc}
s(a) = v_1, & t(a) = v_2 \\
s(b) = v_1 & t(b) = v_3 \\
s(c) = v_2 & t(c) = v_4 \\
s(d) = v_3 & t(d) = v_4
\end{array}
\]
By re-labeling the vertex set of the ladder quiver $\mathbf{A}_3 \times Q$ as 
$$
(\mathbf{A}_3 \times Q)_0 = \{v_1^1,  v_1^2, v_1^3, v_1^4, v_2^1,  v_2^2, v_3^3, v_4^4, v_3^1,  v_3^2, v_3^3, v_3^4, 
v_4^1,  v_4^2, v_4^3, v_4^4\},
$$
we can described it as follows:
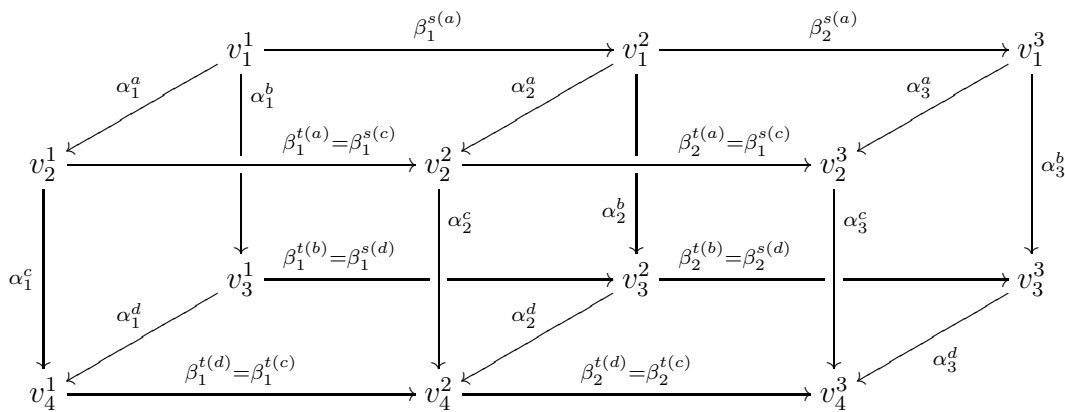
\[
\xymatrix@C=2.0cm{
& v_1^1 \ar[rr]^{\beta_1^{s(a)}} \ar[dl]_{\alpha_1^a} \ar[dd]^(.2){\alpha_1^b}|\hole 
&& v_1^2 \ar[rr]^{\beta_2^{s(a)}} \ar[dd]_(.7){\alpha_2^b}|\hole \ar[dl]_{\alpha_2^a} 
&& v_1^3 \ar[dd]^{\alpha_3^b} \ar[dl]_{\alpha_3^a} \\
v_2^1 \ar[rr]^(.75){\beta_1^{t(a)} = \beta_1^{s(c)}} \ar[dd]_{\alpha_1^c} 
&& v_2^2 \ar[dd]^(.25){\alpha_2^c} \ar[rr]^(.75){\beta_2^{t(a)} = \beta_1^{s(c)}} 
&& v_2^3 \ar[dd]^(0.25){\alpha_3^c} \\
& v_3^1 \ar[rr]^(0.25){\beta_1^{t(b)} = \beta_1^{s(d)}}|\hole \ar[dl]_{\alpha_1^d} 
&& v_3^2 \ar[rr]^(0.25){\beta_2^{t(b)} = \beta_2^{s(d)}}|\hole \ar[dl]_{\alpha_2^d} 
&& v_3^3 \ar[dl]^{\alpha_3^d} \\
v_4^1 \ar[rr]^{\beta_1^{t(d)} = \beta_1^{t(c)}} 
&& v_4^2 \ar[rr]^{\beta_2^{t(d)} = \beta_2^{t(c)}} 
&& v_4^3 
}
\]
}
\end{example}

The filtered objects of an abelian category has been studied in \cite{S99, SS16}. 
We will focus on category of quiver representations. We will also relate filtered 
objects with representations of a ladder quiver.

\begin{definition}\rm{
A filtered representation of $Q$ of length at most $\ell$ is an increasing filtration 
of length $\ell$ in the abelian category of representations of the quiver $Q$ over 
$\mathbb{K}$. We will denote by $M$ the filtered representation, where $M_k$ is a 
representation of $Q$ such that $M_{k - 1} \subseteq M_k$ is a subrepresentation, 
for $k = 2, \ldots , \ell$ .
}
\end{definition}

\begin{remark}\rm{
We can realize the category of representations of quiver $Q$ inside the category of 
filtered representations by taking $M_k = 0$ for $k \leq (\ell - 1)$ (cf. Convention \ref{sink_identification}). The filtered representation can be related to $\Z$-filtered 
representation of Schneiders \cite[Definition 3.1.1]{S99}, by putting $M_k = M_\ell$ 
for $k \geq (\ell + 1)$ and $M_k = 0$ for $k \leq 0$.
}
\end{remark}
 
Next, we shall describe the category of filtered quiver representations as certain 
functor category. This is also useful in getting the structure of quasi-abelian 
category on the category of filtered quiver representations of a given quiver.

Let $\Lambda = \{1, 2, \dots, \ell\}$ be a pre-ordered set. We denote by 
$\mathrm{Fct}(\Lambda, \mathbb{K}Q\text{-mod})$ the category of functors from 
$\Lambda$ to $\mathbb{K}Q\text{-mod}$ \cite{SS16}. We can immediately see the 
following equivalence of categories.

\begin{lemma}\label{lemma-1}
The categories $\mathrm{Fct}(\Lambda, \mathbb{K}Q\text{-mod})$ and $\catchain$ 
are equivalent, where $\catchain$ is the category of representations of $\QAl$ with 
an admissible ideal of relations $I$ \rm{(see Definition \ref{ladder_quiver})}.
\end{lemma}

Let $\mathrm{F}_\Lambda(\mathbb{K}Q\text{-mod})$ be the full subcategory of 
$\mathrm{Fct}(\Lambda, \mathbb{K}Q\text{-mod})$ consisting of filtered objects. 
Then, the natural inclusion functor 
\begin{equation}\label{eq-iota}
\iota\colon \mathrm{F}_\Lambda(\mathbb{K}Q\text{-mod}) \ra 
\mathrm{Fct}(\Lambda, \mathbb{K}Q\text{-mod})
\end{equation}
is an embedding. Moreover, the inclusion functor $\iota$ has a left adjoint
\begin{equation}\label{eq-kappa}
\kappa \colon \catchain  \ra \catfilt
\end{equation}
with $\kappa \circ \iota \simeq \id_{\mathrm{F}_\Lambda(\mathbb{K}Q\text{-mod})}$
(see \cite[Proposition 3.5]{SS16}). We shall use this adjunction in Section 
\ref{sec-det-fns}. We shall also need following basic properties of the category of 
filtered objects.

\begin{proposition}\label{prop: stable kernels}
The subcategory $\mathrm{F}_\Lambda(\mathbb{K}Q\text{-mod})$ of 
$\mathrm{Fct}(\Lambda, \mathbb{K}Q\text{-mod})$ is stable by sub-objects. In particular, 
the category  $\mathrm{F}_\Lambda(\mathbb{K}Q\text{-mod})$ admits kernels and the 
functor $\iota$ commutes with kernels. Moreover, the category 
$\mathrm{F}_\Lambda(\mathbb{K}Q\text{-mod})$ is a quasi-abelian category
\end{proposition}
\begin{proof}
See \cite[Proposition 3.3]{SS16} or \cite[Proposition 3.1.17]{S99}.
\end{proof}

Under the equivalence of Lemma \ref{lemma-1}, to give an object $M$ of
$\mathrm{F}_\Lambda(\mathbb{K}Q\text{-mod})$ 
is equivalent to give a representation $M$ of $\QAl$ such that $M(r) = 0$ for all 
$r\in I$ and $M(a)$ is injective for any arrow in the copies of $\mathbf{A}_\ell$ 
in $\QAl$. We shall denote by $\catfilt$ the full subcategory of $\catchain$ 
consisting of filtered quiver representations of $Q$ of length at most $\ell$. 
We shall denote by $\iota : \catfilt \to \catchain$ the above embedding. 
We get the following corollary of Proposition \ref{prop: stable kernels}.

\begin{cor}
The category $\catfilt$ is a quasi-abelian category. In fact, the category $\catchain$ 
can be identified with the (left) abelian envelope of $\catfilt$ via the functor $\iota$.
\end{cor}
\begin{proof}
The proof follows from \cite[Theorem 3.9]{SS16} (see also \cite[Proposition 3.1.17]{S99}) 
for the first part, and from \cite[Corollary 3.1.29]{S99} for the second part. 
\end{proof}

By \cite[Proposition 1.2.14]{An09}, there exists a full strict subcategory $\mathcal{T}$ 
of $\catchain$ such that for every object $B\in \catchain$, there exists 
$B_{\mathrm{tor}} \in \mathcal{T}$ and $M \in \catfilt$ and a short exact sequence
\begin{equation} \label{cotilting torsion extension}
0\ra B_{\mathrm{tor}}\ra B\ra M\ra 0 \;.
\end{equation}

Moreover, if $\mathrm{rk}(B) = 0$, then $B \in \mathcal{T}$ using \cite[Remark 1.2.15]{An09}.

\begin{lemma}\label{tf resolution}
There is a short exact sequence for any object $B$ of $\catchain$,
\[
 0 \to N' \to N \to B \to 0
\]
where $N$ and $N'$ are filtered representations.
\end{lemma}
\begin{proof}
To get the epimorphism from filtered representation $f : N \to B$, define
$$
N(k , v) := \bigoplus_{j = 1}^k B(j , v)
$$
and, vertical maps are defined as in $B$ and horizontal maps are the canonical inclusions. 
The epimorphism $f_{(k , v)}$ is defined using sum of maps
$B(j , v) \to B(k , v)$ defined by 
$$
B(\beta^v_{k-1}) \circ \cdots \circ B(\beta^v_j)\,,
$$
whenever $j < k$ and for $j=k$ take identity map on that component.
Now, we can check that this gives a surjection from $N$ to $B$. Take $N' := \ker f$. 
Then, $N'$ is again a filtered representation. We can check that $N'(1 , v) = 0$ and 
for $k >1$, $N'(k , v)$ is a vector subspace generated by
$$
(0, \ldots, 0, -b_j, 0, \ldots, 0, [B(\beta^v_{k-1}) \circ \cdots \circ B(\beta^v_j)] (b_j) )
$$ 
for $j = 1, \ldots, k-1$ and $b_j \in B_j$.
\end{proof}

We can now describe the category $\mathcal{T}$ explicitly using the Convention 
\ref{sink_identification}.

\begin{proposition}
Let $B$ be an object of $\catchain$.
The representation $B$ is in $\mathcal T$ if and only if $B(\ell , v) = 0$ for each 
vertices $(\ell , v)$ of the quiver $\QAl$.
\end{proposition}
\begin{proof}
We can see that all objects with $B(\ell , v)=0$ for vertices $(\ell , v)$ satisfies the 
conditions of cotilting torsion pair. That is, there is an epimorphism from filtered 
representation to $B$ and there is a short exact sequence \ref{cotilting torsion extension}, 
as orthogonality $\Hom{B_{\mathrm{tor}}}{M} = 0$ is easy to check. The epimorphism follows 
from Lemma \ref{tf resolution}.

Define 
\[
B_{\mathrm{tor}}(w) = 
\left\{\begin{array}{ll}
\ker (B(\beta^v_{\ell-1}))\,, & w = (\ell - 1, v) \\

0\,, & w = (\ell, v) \\

[B(\beta^v_{\ell-1}) \circ \cdots \circ B(\beta^v_j)]^{-1}(B_{\mathrm{tor}}(\ell-1 , v))\,, & 
w = (j, v)\;, 1\leq j < \ell -1
\end{array}\right.
\]
We can check that this will define a torsion sub-object and quotient of this will be 
filtered representation. This completes the proof.
\end{proof}

Recall, the \emph{kernel} of a morphism $f : M \to N$ in the category $\catfilt$ 
is defined as $\ker(f)_{(j, v)} := \ker(f_{(j, v)})$ with induced maps for each arrows. 
The \emph{cokernel} of a morphism $f : M \to N$ in the category $\catfilt$ is defined as
$$
\mbox{coker}(f)_{(j, v)} := \mbox{Im}(M((j, v))\ra \mbox{coker}(f_{\ell, v}))
$$
with induced maps for each arrows (see \cite[Corollary 3.6]{SS16} or 
\cite[Proposition 3.1.2]{S99}). In case, a sub-object (respectively, quotient object) 
is kernel (respectively cokernel) of a morphism, then it is called a strict monic 
(respectively, strict epi). 

We can observe that  there are  monic in $\catfilt$ 
which are not strict monic. We can check that the monic morphism $f$ is strict monic, 
if $\iota(\mbox{coker}(f)) = \mbox{coker}(\iota(f))$.

More generally, we get the following definition of \emph{strict} filtration using the 
concept of \emph{strict} monic.

Let $M$ be an object of $\catfilt$. A filtration
$$
0 = M_0 \subset M_1 \subset M_2 \subset \cdots \subset M_{k-1} \subset M_k = M
$$ 
of $M$ is called \emph{strict} filtration if each inclusion maps are \emph{strict} 
monic morphisms of $\catfilt$ i.e. each quotient satisfies 
$\iota(M_i / M_{i-1}) = \iota(M_i)/\iota(M_{i-1}) $.

We recall the following definition of \cite[Definition 1.2.6]{An09}.
\begin{definition}\label{maximal flag} \rm{
A strict filtration 
$$
0 = M_0 \subset M_1 \subset M_2 \subset \cdots \subset M_{k-1} \subset M_k = M
$$ 
is called a flag of length $k$ on $M$ if $M_i \neq M_{i - 1}$ for $1 \leq i \leq k$. 
A flag of maximal length (on $M$) is called the maximal flag (on $M$).
}
\end{definition}
Notice that, a flag is \emph{maximal} if all the quotients $M_{i + 1} / M_i$ have no 
proper nonzero strict sub-objects. Hence, in particular, if all the quotients of a 
flag are simple objects then it is a maximal flag. Now, using lemma \cite[Lemma 1.2.8]{An09}, 
we can prove that any object $M$ of $\catfilt$ has a unique maximal flag on $M$. 
Therefore, we get a well defined associated graded object of a maximal flag on $M$
\begin{equation}\label{max gr}
\mathrm{gr^{max}}(M):= \bigoplus_{i= 1}^k M_i/M_{i-1}.
\end{equation}

We now recall definition of projective objects which will be used later in Section 
\ref{sec-det-fns}.

\begin{definition} 
A filtered representation $P$ is called projective (respectively, strongly projective) 
if for any strict (respectively, arbitrary) epimorphism $ M \to  N$ the associated map
$\Hom[\catfilt]{P}{M}  \to \Hom[\catfilt]{P}{N} $ is surjective.
\end{definition}
Dually, we can also define the injective objects in $\catfilt$, see 
\cite[Section 1.3.4]{S99}.
The projective objects in an exact category is related to the projective objects 
of its abelian envelope. In particular, for filtered representations, we get the 
following results.

\begin{proposition}\label{filtered projective}\cite[Proposition 1.3.24]{S99}
An object $P$ of $\catfilt$ is projective if and only if $\iota(P)$ is projective in 
$\catchain$. The category $\catfilt$ has enough projectives and any projective object of 
$\catchain$ is same as $\iota (P)$ for some projective object $P$ of $\catfilt$.
\end{proposition}
\begin{proof}
Since $\catfilt$ is a strictly full subcategory and the left abelian envelope $\catchain$ 
has enough projective objects, the assertions immediately follow from 
\cite[Proposition 1.3.24]{S99}.
\end{proof}

Recall that the indecomposable projective modules in $\catchain$ are in bijection 
with the vertices of $\QAl$ and the module corresponding to a vertex $w$ is just 
$P'_w = P_w/IP_w$, where $P_w$ is an indecomposable projective representation of 
$\QAl$ corresponding to a vertex $w$ in $\QAl$. The indecomposable projective 
modules in $\catchain$ are characterized by the property that for each object 
$N'\in \catchain$ and for each vertex $v$ of $\QAl$, we have
$$
\Hom[\catchain]{P'_w}{N'} = N'(w),
$$ 
where $N'(w)$ is a vector space assigned to a vertex $w$ by the representation 
$N'$ of $\QAl$. Now using the Proposition \ref{filtered projective}, the objects 
$P_w'$ will become the projective indecomposable objects of $\catfilt$. 
More precisely, the objects $P_w'$ are filtered representations, and they give the 
complete set of projective indecomposable objects of $\catfilt$.

\section{Rank function and slope stability}\label{sec-rank-fn}
In this section, we first discuss the notion of degree and rank functions on the 
category of filtered quiver representations following \cite{An09}. The concept 
degree and rank functions provides a good definition of semistability for filtered 
quiver representations to construct the moduli of filtered quiver representations 
using GIT. More precisely, the corresponding slope semistability for 
filtered quiver representations provides a good description of the closed points 
of the moduli scheme. At the end of the section, we also give two kinds of slope 
filtrations for which the corresponding moduli schemes have quite a different 
description of closed points.

It is known that the Grothendieck group $K_0(\mathbb{K}Q)$ of the path algebra 
$\mathbb{K}Q$ is a free abelian group of finite rank, and the rank of $K_0(\mathbb{K}Q)$ 
is precisely the cardinality of $Q_0$.
In fact, the algebra $\mathbb{K}Q$ is a hereditary algebra and simple modules are 
in bijection with the vertex set $Q_0$. We get $K_0(\mathbb{K}Q) \simeq \Z Q_0$, 
where $\Z Q_0$ is the free abelian group generated by the set $Q_0$. Hence, we can 
identify an equivalence class of each $\mathbb{K}Q$-module $M$ in $K_0(\mathbb{K}Q)$ 
with the dimension vector of corresponding representation of the underlying quiver $Q$, 
which we denote by $\dv(M)$. 
Similar identification exists for any quiver with admissible ideal of relations 
\cite[Section 4]{Ki94}. We can also identify the set of all additive functions 
$\Hom{K_0(\mathbb{K}Q)}{\Z}$ with $ \Z^{Q_0}$. These additive functions are used to 
give the definition of $\theta$-stability  by A. King \cite{Ki94}. 

Since $\catfilt$ is a quasi-abelian category, using a result of Schneider 
\cite[Proposition 1.2.35]{S99}, we can get the isomorphism 
$$
K_0(\iota): K_0(\catfilt) \simeq K_0(\catchain).
$$
Hence, we can identify $\Hom{K_0(\catfilt)}{\Z}$ with $ \Z^{Q_0}$ too, and therefore 
any additive function on filtered representations is in bijection with additive 
functions on representations of the ladder quiver.

Let $\mathrm{sk}(\catfilt)$ be the skeleton of $\catfilt$, that is, the set of 
isomorphism classes of objects of $\catfilt$.

\begin{definition}\cite[Def. 1.2.9]{An09}\label{def-rk-fn}\rm{
A \emph{rank function} on $\catfilt$ is a function 
$$
\mathrm{rk}\colon \mathrm{sk}(\catfilt) \ra \mathbb{N}
$$
which is additive on short exact sequences and takes the value $0$ only on the 
zero object.
}
\end{definition}
The rank function will extend to give the additive map 
$$
\mathrm{rk} \colon K_0(\catfilt) \to \Z ; (d_w) \mapsto \sum_{w\in (\QAl)_0} r_w d_w \, 
$$
with  positivity conditions, $\sum_{j = k}^l r_{(j, v)} > 0$  for each $1 \leq k \leq l$. 
Clearly, if all $r_w > 0$ then these conditions are satisfied and we get the notion of 
positive additive function in the sense of Rudakov \cite{Ru97}.

Given a $\Theta\in \Gamma := \Z^{(\QAl)_0}$, we get a function $\mathrm{sk}(\catfilt)\ra \Z$ 
given by
$$
\Theta(\dv(M)):= \sum_{w\in (\QAl)_0} \Theta_w \dv(M)_w \,.
$$
The additive function $\Theta$ is also called degree function. Hence, we get a group 
homomorphism
$$
\Theta \colon K_0(\catfilt)\ra \Z\,.
$$

We define a slope function $\mu \colon \mathrm{sk}(\catfilt)-\{0\}\ra \mathbb{Q}$ 
as follows:
$$
\mu_{\Theta, \mathrm{rk}}(M):= \frac{\Theta(\dv(M))}{\mathrm{rk}(\dv(M))}\,.
$$

\begin{definition}\rm{
We say that a representation $M$ of a quiver $Q$ is $\mu_{\Theta, \mathrm{rk}}$-semistable, 
if for all non-zero subrepresentations $M'$ of $M$, we have 
$$
\mu_{\Theta, \mathrm{rk}}(\dv(M'))\leq \mu_{\Theta, \mathrm{rk}}(\dv(M))
$$
If the inequality is strict for all proper non-zero subrepresentations $M'$ of $M$, 
then we say that $M$ is $\mu_{\Theta, \mathrm{rk}}$-stable. 
}
\end{definition}

Since $K_0(\catfilt) \simeq K_0(\catchain)$, we get an extension of the rank function 
$\mathrm{rk}$ to $K_0(\catchain)$. This extension of rank function may take value $0$ 
on some non-zero objects in $\catchain$. In view of this, we may not get slope function 
on $K_0(\catchain)$ as an extension of $\mu_{\Theta, \mathrm{rk}}$. However, we can 
choose an appropriate $\theta\in \Gamma$ so that we can relate the 
$\mu_{\Theta, \mathrm{rk}}$-stability with the $\theta$-stability (defined in \cite{Ki94}).

Recall that for given any stability parameter $\theta\in \Gamma$, an object $M$ of 
$\catchain$ is called $\theta$-semistable (respectively, $\theta$-stable) if 
$\theta(\dv(M)) = 0$ and for any non-trivial sub-object $M'$ of $M$, we have 
$\theta(\dv(M'))\geq 0$ (respectively, $\theta(M') > 0$).

Fix a $\Theta\in \Gamma$ and a dimension vector $\dv$.  Let $\mathrm{rk}$ be a rank 
function as in the Definition \ref{def-rk-fn}. Once we fix the slope 
$\mu = \Theta(\dv) / \mathrm{rk}(\dv)$, then we can define 
$$
\theta := (\theta_w)_{w\in (\QAl)_0} = 
(\Theta(\dv)r_w - \Theta_w \mathrm{rk}(\dv))_{w\in (\QAl)_0}.
$$

\begin{proposition}
Any object $M$ of $\catfilt$ having slope $\mu$ is $\mu_{\Theta,\mathrm{rk}}$-semistable 
(respectively, $\mu_{\Theta,\mathrm{rk}}$-stable) if and only if $ \iota(M)$ is 
$\theta$-semistable (respectively, $\theta$-stable). 
\end{proposition}
\begin{proof}
Let $M$ be an object of $\catfilt$ having slope $\mu$. If $M$ is 
$\mu_{\Theta,\mathrm{rk}}$-semistable , then
for any non-zero subobject $M'$ of $M$, we have 
$\mu_{\Theta,\mathrm{rk}}(M') \leq \mu_{\Theta,\mathrm{rk}}(M)$.
That is, we have
$$
\frac{\Theta(\dv(M'))}{\mathrm{rk}(\dv(M'))} \leq 
\frac{\Theta(\dv(M))}{\mathrm{rk}(\dv(M))} = \mu = 
\frac{\Theta(\dv)}{\mathrm{rk}(\dv)}\,.
$$
Now, 
\[
\begin{array}{ll}
\theta(M')  
& = \displaystyle \sum_{w\in (\QAl)_0} (\Theta(\dv)r_w - \mathrm{rk(\dv)\Theta_w})\dv(M')_w \\
& \\
& = \Theta(\dv)\mathrm{rk}(\dv(M')) - \mathrm{rk}(\dv)\Theta(\dv(M'))\\
& \\
& \geq 0\,.
\end{array}
\]
In view of Proposition \ref{prop: stable kernels}, we can conclude that $ \iota(M)$ 
is $\theta$-semistable. Conversely, if $ \iota(M)$ is $\theta$-semistable, then by 
reversing the above arguments, we can see that $M$ is $\mu_{\Theta,\mathrm{rk}}$-semistable.
\end{proof}

\begin{proposition}\label{abelian-condition}
Given $\Theta\in \Gamma$ such that the degree function is non-negative on torsion 
subcategory, the full subcategory $\catfilt(\mu)$ consisting of zero object and the 
$\mu_{\Theta, \mathrm{rk}}$-semistable filtered representations of fixed
slope $\mu$ is a quasi-abelian category.
\end{proposition}
\begin{proof}
Assume that $\Theta$ is chosen in such a way that for each epi-monic $M \ra M'$, 
we have $\mu_{\Theta, \mathrm{rk}}(M) \leq \mu_{\Theta, \mathrm{rk}}(M')$. Then, we 
get a slope function in the sense of \cite[ Definition 1.3.1]{An09}. This is also 
equivalent to choosing a degree function which is non-negative on objects of torsion 
subcategory, see \cite[Corollary 1.4.10]{An09}. The result follows using 
\cite[Lemma 1.3.9]{An09}.
\end{proof}
 
Recall that the $\theta$-semistable objects in $\catchain$ forms an abelian full 
subcategory of $\catchain$. It is both Noetherian and Artinian, and hence we 
have Jordan-H\"older filtration for each $\theta$-semistable object in $\catchain$.

\begin{proposition} \label{strict JH}
If there is a strict filtration of any object $M$ in $\catfilt(\mu)$ with the 
successive quotients $\mu_{\Theta,\mathrm{rk}}$-stable, then the filtration obtained 
after applying $\iota$ is a Jordan-H\"older filtration of $\iota(M)$.
\end{proposition}
\begin{proof}
Since the filtration is a strict filtration, and hence the successive cokernels are 
preserved under the functor $\iota$. We can observe that the category $\catfilt(\mu)$ 
is closed under sub-objects inside $\catchain$, where the stability on $\catchain$ 
is given by an appropriate $\theta$. Hence, $\mu_{\Theta, rk}$-stable objects in 
$\catfilt$ goes to $\theta$-stable objects in $\catchain$ via the functor $\iota$. 
Now using the fact that the stable objects are simple objects and hence after applying 
$\iota$, we get Jordan-H\"older filtration.
\end{proof}

\begin{definition}\rm{
Given an object $M\in \catfilt(\mu)$, a strict filtration
$$
0 = M_0 \subset M_1 \subset M_2 \subset \cdots \subset M_{k-1} \subset M_k = M
$$
such that all the quotients $M_i/M_{i-1}$ are $\mu_{\Theta, \mathrm{rk}}$-stable 
objects in $\catfilt$ having the same slope, is called the strict Jordan-H\"older 
filtration. 
}
\end{definition}
Since any stable object is nonzero, a strict Jordan-H\"older filtration gives an 
example of maximal flag (see Definition \ref{maximal flag}) in the category 
$\catfilt(\mu)$. Using the Proposition \ref{strict JH}, the Jordan-H\"older theorem 
is valid for strict Jordan-H\"older filtrations in $\catfilt(\mu)$. 
Hence, we get the well defined associated graded objects of a strict Jordan-H\"older 
filtration. The associated graded object of a strict Jordan-H\"older filtration is 
defined as 
$$
\mathrm{gr}(M):= \bigoplus_{i= 1}^k M_i/M_{i-1}
$$
and it depends only on $M$.
An object $N$ is called $\mu_{\Theta, \mathrm{rk}}$-polystable
if $N \simeq \mathrm{gr}(M)$ for some object $M$ with a strict Jordan-H\"older filtration.
\begin{definition}\rm{
We say that two objects $M$ and $N$ in $\catfilt(\mu)$ are $S^{JH}$-equivalent if 
both objects have strict Jordan-H\"older filtration and the associated graded 
objects $\mathrm{gr}(M)$ and $\mathrm{gr}(N)$ are isomorphic. We say that $M$ and 
$N$ in $\catfilt(\mu)$ are $S$-equivalent if $\mathrm{gr^{max}}(M)$ and 
$\mathrm{gr^{max}}(N)$ are isomorphic.
}
\end{definition}

Now using the above definition, we get the following Corollary of Proposition \ref{strict JH}
\begin{cor} \label{cor: gr comm}
The restriction of the functor $\iota$ \rm{(see \eqref{eq-iota})} to 
$\catfilt(\mu)$ takes $S^{JH}$-equivalent objects to $S$-equivalent objects. 
Moreover, $\iota (\mathrm{gr}(M)) \simeq \mathrm{gr}(\iota(M))$.
\end{cor}
\begin{proof}\label{rem-s-equivalence}
Let $M$ and $N$ be two $S^{JH}$-equivalent objects in $\catfilt(\mu)$. Since 
$\mu_{\Theta,\mathrm{rk}}$-stability and $\theta$-stability are same for objects 
of $\catfilt$, we can define S-equivalence with respect to $\theta$-stability for 
objects of $\catfilt(\mu)$. Using Proposition \ref{prop: stable kernels}, we know 
that sub-objects, and hence filtrations of $M$ (respectively, $N$) coincide via 
$\iota$, and hence Jordan-H\"older filtration does not change under $\iota$. Hence, 
if we view the objects $M$ and $N$ in $\catchain$, then they remain 
$S$-equivalent in $\catchain$ with respect to the notion of $\theta$-stability. 
\end{proof}

\begin{remark}\rm{
We should mention here that it may be possible that $M$ and $N$ are not $S^{JH}$-equivalent in 
$\catfilt(\mu)$, but they are $S$-equivalent in $\catchain$. This is also reflected in the 
GIT picture (see Section \ref{sec-git-mc}).
}
\end{remark}

\begin{remark}\rm{
Let $\Theta \in \Gamma$ as in Proposition \ref{abelian-condition}, then there exists a functorial slope decreasing strict 
filtration (slope filtration \cite[Proposition 1.4.6, Theorem 1.4.7]{An09}) which will be called Harder-Narasimhan filtration 
(or HN-filtration) of filtered representations. Moreover, 
the HN-filtration of $\iota(M)$ is same as image of HN-filtration of $M$.
}
\end{remark}

\begin{example}\rm{
Define a rank function
$\mathrm{rk}_{\dim}\colon \mathrm{sk}(\catfilt) \ra \Z$
by 
$$
\mathrm{rk}_{\dim} (M):= \sum_{v\in Q_0}\dim M(v) + 
\sum_{(j, v) \in \mathbf{A_{\ell 1}} \times Q_0} \mathrm{rank}  \beta_j^v\,
$$
where $M(v) := M(\ell, v)$ (see Convention \ref{sink_identification}).

In other words, we have $r_w = 1$ for all $w\in (\QAl)_0$. The extension of this rank function 
$\mathrm{rk}_{\dim}$ on $K_0(\catchain)$ is precisely the $\dim$ function which is given by
$$
\dim (\dv) :=  \sum_{w \in (\QAl)_0} d_w \,.
$$
This can also be seen using the description of $N$ and $N'$ as in Lemma \ref{tf resolution} by 
computing the dimensions, namely 
$$
\dim N = \sum_{j = 1}^\ell (\ell - j +1) \dim B_j
$$
and 
$$
\dim N' = \sum_{j = 1}^\ell (\ell - j) \dim B_j.
$$
Using the additivity of $\mathrm{rk}_{\dim}$, we can check that $\mathrm{rk}_{\dim}(B) = \dim (B)$.

In particular, $\mathrm{rk}_{\dim}(B) = \dim (B) = 0$ if and only if $B \simeq 0$ for each object $B$ of $\catchain$.

Since the function $\mathrm{rk}_{\dim}$ is nonzero on some torsion classes, we can not use 
\cite[Corollary 1.4.10]{An09} to get the abelian category structure on $\catfilt(\mu)$.

}
\end{example}

\begin{example}\label{ex:sink slope}\rm{
We can define another positive additive function given by dimension at vertices $(\ell , v)$, say 
$\mathrm{rk}_s (B) := \sum_{v \in Q_0} r_v \dim B(\ell , v)$, where $r_v > 0$.
The arguments given in the above example (using Lemma \ref{tf resolution}) shows that $\mathrm{rk}_s (B) = 0$ 
if and only if $ B$ is isomorphic to an object of 
$\mathcal T$. We can also define the slope using this new positive additive function, say 
$\mu_s (B) := \Theta(\dv(B)) / \mathrm{rk}_s(B)$. Now using \cite[Corollary 1.4.10]{An09}, we can get the 
abelian category structure on $\catfilt(\mu)$ with respect to the stability $\mu_s$, if $\Theta\in \Gamma$ is
chosen as in the Proposition \ref{abelian-condition}. Hence, in this case, the Jordan-H\"older
theorem holds in $\catfilt(\mu)$. Consequently, the notion of $S^\mathrm{JH}$-equivalence and $S$-equivalence
coincides, in such cases.
}
\end{example}

\begin{remark}\rm{
If rank function is zero for all objects of torsion subcategory, then it is of the form $\mathrm{rk}_s$ of example \ref{ex:sink slope}. 
}
\end{remark}

\section{GIT and Moduli construction}\label{sec-git-mc}
In this section, we consider the moduli problem for filtered quiver 
representations with the slope stability using rank function 
(see Definition \ref{def-rk-fn}).

We first recall some basic definitions.
Let $\Sch^\mathrm{op}$ be the opposite category of $\mathbb{K}$-schemes and $\Set$ 
be the category of sets. For a scheme $Z$, its functor of points
$$
\underline{Z}\colon \Sch^\mathrm{op}\ra \Set
$$
is given by $X\mapsto \mathrm{Hom}(X, Z)$. By Yoneda Lemma, every natural transformation
$\underline{Y}\ra \underline{Z}$ is of the form $\underline{f}$ for some morphism of schemes
$f\colon Y\ra Z$.

\begin{definition}\cite[\S 1, p. 60]{Si94}\rm{
Let $\M \colon \Sch^\mathrm{op}\ra \Set$ be a functor, $\Ms$ a scheme and 
$\psi\colon \M\ra \underline{\Ms}$ a natural transformation. We say that 
$(\Ms, \psi)$ co-represents $\M$ if for each scheme $Y$ and each natural 
transformation $h\colon \M\ra \underline{Y}$, there exists a unique $g\colon \Ms\ra Y$
such that $h = \underline{g}\circ \psi$, that is the following diagram
\[
\xymatrix{
\M \ar[d]_\psi \ar[rd]^h \\
\underline{\Ms} \ar[r]_{\underline{g}} & \underline{Y}
}
\]
commutes.
}

Let $\M_1, \M_2 \colon \Sch^\mathrm{op}\ra \Set$ be two functors. A morphism of 
functors $g\colon \M_1 \ra \M_2$ is said to be a local isomorphism, if it induces 
an isomorphism of sheafification in the Zariski topology.
\end{definition}

Note that if $g\colon \M_1 \ra \M_2$ is a local isomorphism, then the condition that 
$(\Ms_1, \psi_1)$ co-represents $\M_1$ is equivalent to the condition that 
$(\Ms_2, \psi_2)$ co-represents $\M_2$ \cite[\S 1, p. 60]{Si94} or \cite[Lemma 4.7]{AK07}.

\subsection{Representation spaces and GIT}\label{subsec-git}
Let us fix $\Theta\in \Gamma$ and the dimension vector $\dimv{d}$ for the quiver 
$\QAl$. Let 
$$
\rs := \bigoplus_{a\in (\QAl)_1} 
\Hom[\mathbb{K}]{\mathbb{K}^{d_{s(a)}}}{\mathbb{K}^{d_{t(a)}}}
$$ 
be the space of representations of $\QAl$ having dimension vector $\dimv{d}$.
Let $\rsc$ be the closed subset of $\rs$ consisting of $(\alpha_a)$ satisfying the 
relation in $I$. Let $\rsf$ be an open subset of $\rsc$ consisting of 
$(\alpha_a)\in \rsc$ such that for any $a\in \mathbf A_{\ell 1} \times Q_0$, we have 
$\alpha_a$ injective map.

There is a natural action of the group 
$$
G(\QAl):= \displaystyle \prod_{w\in (\QAl)_0} \mathrm{GL}(\mathbb{K}^{d_w})
$$
on $\rs$ such that the isomorphism classes in $\catQAl$ correspond to the orbits in 
$\rs$ with respect to this action. Moreover, $\rsc$ and $\rsf$ are invariant under 
this action. Note that $\rsc$ is a closed subscheme of $\rs$, while $\rsf$ is a 
locally closed subscheme of $\rs$.
Let $\rs^\mathrm{ss}$ (respectively, $\rsc^\mathrm{ss}$, $\rsf^\mathrm{ss}$) be the 
open subset of $\rs$ (respectively, $\rsc$, $\rsf$) which is 
$\mu_{\Theta, \mathrm{rk}}$-semistable locus in $\rs$ (respectively, $\rsc$, $\rsf$).

Let $G:= G(\QAl)/\Delta$, where
$\Delta:= \{(t\mathbf{1}_{\mathbb{K}^{d_w}})_{w\in (\QAl)_0}\,|\, t\in \mathbb{K}^*\}$ 
and there is an action of this group $G$ on $\rs$.

Let $\chi_{\theta} \colon G\ra \mathbb{K}^*$ be the character defined by
\begin{equation}\label{eq-chi-theta}
\chi_{\theta}((g_w)):= \prod_{w\in (\QAl)_0} 
\det(g_w)^{\big(\Theta(\dv)r_w - \mathrm{rk}(\dv)\Theta_w\big)}
\end{equation}
where $\Theta\in \Gamma$. 
\begin{proposition}
As a special case of \cite[Proposition 3.1 \& 3.2]{Ki94}, we have the following:
\begin{enumerate}
\item A point $x\in \rs$ is $\chi_{\theta}$-semistable if and only if the corresponding object 
$M_x\in \catQAl$ is $\mu_{\Theta, \mathrm{rk}}$-semistable.
\item Two points $x, y \in \rs^\mathrm{ss}$ are GIT-equivalent (with respect to $\chi_{\theta}$) 
if and only if the corresponding objects $M_x$ and $M_y$ are $S$-equivalent in $\catQAl$ 
(with respect to $\theta$-stability).
\item A point $x\in \rsc$ is $\chi_{\theta}$-semistable if and only if the corresponding object 
$M_x\in \catchain$ is $\mu_{\Theta, \mathrm{rk}}$-semistable.
\item Two points $x, y \in \rsc^\mathrm{ss}$ are GIT-equivalent (with respect to $\chi_{\theta}$) 
if and only if the corresponding objects $M_x$ and $M_y$ are $S$-equivalent in $\catchain$ 
(with respect to $\theta$-stability).
\end{enumerate}
\end{proposition}

\subsection{Moduli functors}

A flat family of $\mathbb{K}(\QAl)$-modules over a connected scheme $S$ is a 
locally-free sheaf $\sheaf{F}$ over $S$ together with a $\mathbb{K}$-algebra 
homomorphism $\mathbb{K}(\QAl)\ra \mathrm{End}(\sheaf{F})$.
On the other hand, a flat family of representations of $\QAl$ is a representation 
of $\QAl$ in the category of locally-free sheaves over $S$. The equivalence between 
the $\mathbb{K}(\QAl)$-modules and the representations of $\QAl$ extends naturally 
to families.

Consider the functor
$$
\mf(\dv) \colon \Sch^\mathrm{op}\ra \Set ~(\mbox{respectively, } 
\mtf(\dv) \colon \Sch^\mathrm{op}\ra \Set)
$$
defined by assigning to each $\mathbb{K}$-scheme $S$ the set of isomorphism classes 
of flat families over $S$ of $\mu_{\Theta, \mathrm{rk}}$-semistable  
(respectively, $\mu_{\Theta, \mathrm{rk}}$-stable) representations of $\QAl$ having 
dimension vector $\dimv{d}$. 

There is a natural functor $h\colon \fp{\rs^\mathrm{ss}} \ra \mf(\dv)$ 
(defined by $(f\colon S\ra \rs^\mathrm{ss}) \mapsto [f^*\mathbb{M}]$, where $\mathbb{M}$ 
is a tautological family on $\rs^\mathrm{ss}$), which induces a local isomorphism 
$\tilde{h}\colon \fp{\rs^\mathrm{ss}}/\fp{G} \ra \mf(\dv)$. 
This reduces the problem to the existence of good quotient of $\rs^\mathrm{ss}$ by $G$. 
It is proved in \cite{Ki94} that the GIT quotient $\pi \colon \rs^\mathrm{ss}\ra \ms(\dv)$ 
exists and it is a good quotient.

\begin{theorem}\cite{Ki94}
There exist a projective variety $\ms(\dv)$ over $\mathbb{K}$ which co-represents 
the moduli functor $\mf(\dv)$. In particular, the closed points of $M$ corespond to 
the $S$-equivalence classes of representations of $\QAl$ having dimension vector $\dimv{d}$. 

Moreover, there exists a quasi-projective variety $\mst(\dv)$ co-representing the moduli 
functor $\mtf(\dv)$ whose closed points correspond to isomorphism classes of stable 
representations.
\end{theorem}

Consider the functor
$$
\mfc(\dv) \colon \Sch^\mathrm{op}\ra \Set ~(\mbox{respectively, } 
\mtfc(\dv) \colon \Sch^\mathrm{op}\ra \Set)
$$
defined by assigning to each $\mathbb{K}$-scheme $S$ the set of isomorphism classes 
of flat families over $S$ of $\mu_{\Theta, \mathrm{rk}}$-semistable 
(respectively, $\mu_{\Theta, \mathrm{rk}}$-stable) objects of $\catchain$ having 
dimension vector $\dimv{d}$.

By restricting the tautological family $\mathbb{M}$ on $\rs^\mathrm{ss}$ to 
$\rsc^\mathrm{ss}$, we get a natural functor $\fp{\rsc^\mathrm{ss}} \ra \mfc(\dv)$. 
By \cite[Proposition 5.2]{Ki94},  it induces a local isomorphism 
$\tilde{h}_\mathrm{rel}\colon \fp{\rsc^\mathrm{ss}}/\fp{G} \ra \mfc(\dv)$ reducing 
the problem to the existence of good quotient $\rsc^\mathrm{ss}$ by $G$. It is proved 
in \cite[\S 4]{Ki94} that the the GIT quotient 
$\pi_\mathrm{rel} \colon \rsc^\mathrm{ss}\ra \msc(\dv)$ exists and it is a good quotient.

\begin{theorem}\cite{Ki94}
There exist a projective variety $\msc(\dv)$ over $\mathbb{K}$ which co-represent the 
moduli functor $\mfc(\dv)$. In particular, the closed points of $\msc(\dv)$ correspond 
to the $S$-equivalence classes of objects of $\catchain$ having dimension vector $\dimv{d}$.

Moreover, there exists a quasi-projective variety $\mstc(\dv)$ co-representing the 
moduli functor $\mtfc(\dv)$ whose closed points correspond to isomorphism classes 
of stable representations.
\end{theorem}

The following result stated for ladder quiver, however, remains valid for any finite 
quiver (without oriented cycles) with admissible ideal.

\begin{proposition}\label{prop: rel embedding}
There is a closed set-theoretic embedding $\psi \colon \msc(\dv) \ra \ms(\dv)$. If 
characteristic of $\mathbb{K}$ is zero, then $\psi$ is closed scheme-theoretic embedding. 
In positive characteristic, the morphism $\psi$ is scheme-theoretic embedding on the 
stable locus $\mstc(\dv)$.
\end{proposition}
\begin{proof}
Since $\pi_{\rm rel}$ is a categorical quotient, we get the following commutative 
diagram of topological spaces
\[
\xymatrix{
 \rsc^\mathrm{ss} \ar[d]_{\pi_\mathrm{rel}} \ar[r]^j
& \rs^\mathrm{ss} \ar[d]^\pi \\
 \msc(\dv) \ar[r]_\psi & \ms(\dv) 
}.
\]
Since $j$ is a closed embedding, $\pi$ and $\pi_{\rm rel}$ are good quotients, we can 
see that the map $\psi$ is a closed embedding.
 
Note that $\msc(\dv)  = \pi(\rsc^\mathrm{ss})$ in characteristic zero, using Reynold's 
operator and in characteristic $p$, we can use properties of geometric quotient on 
stable locus to get the scheme theoretic embedding (cf. \cite[Proposition 6.7]{AK07}).
\end{proof}

Now, using $\mu_{\Theta, \mathrm{rk}}$-semistability, we define moduli 
functor for filtered quiver representations. Consider the functor
$$
\mff(\dv) \colon \Sch^\mathrm{op}\ra \Set ~(\mbox{respectively, } 
\mtff(\dv) \colon \Sch^\mathrm{op}\ra \Set)
$$
defined by assigning to each $\mathbb{K}$-scheme $S$ the set of isomorphism classes of 
flat families over $S$ of $\mu_{\Theta, \mathrm{rk}}$-semistable (respectively, 
$\mu_{\Theta, \mathrm{rk}}$-stable) objects of $\catfilt$ having dimension vector 
$\dimv{d}$.

In the remaining part of this section, we reduce the problem of co-representability 
of this moduli 
functor to the existence of GIT quotient, and describe the closed points of the 
corresponding moduli space (co-representing scheme) using the results of previous 
sections.

\subsection*{Moduli construction}
The following Proposition is key to give an algebraic description for two points in the 
representation space $\rsf$ to be GIT equivalent (see Proposition \ref{prop:S-eq}).

\begin{proposition}\label{prop-1-ps-filtration} Let $x\in \rsf$. Then,
there is a surjection from the set
$$
\{1\text{-}PS\; \lambda\colon \mathbb{K}^*\ra G\; \mbox{such that}\; 
\lim \lambda(t)\cdot x\; \mbox{exist in}\; \rsf \}
$$
to
$$
\{\mbox{strict filtrations of}\; M_x\; \mbox{in}\; \catfilt \} .
$$
\end{proposition}
\begin{proof}
Let $\lambda\colon \mathbb{K}^*\ra G$ be an one-parameter subgroup of $G$ such that the 
$\lim_{t\rightarrow 0} \lambda(t)\cdot x = y$ exist in $\rsf$. Following \cite{Ki94}, 
we have a weight space decomposition 
$$
M_x(w) = \bigoplus_{n\in \Z} M_x(w)^n
$$
for each vertex $w\in (\QAl)_0$, where $\lambda(t)$ acts on the weight space $M_x(w)^n$ 
as multiplication by
$t^n$. For each arrow $a$, we have $M_x(a) = \bigoplus M_x(a)^{mn}$, where 
$$
M_x(a)^{mn}\colon M_x(s(a))^n\ra M_x(t(a))^m\,.
$$
Moreover, we have 
$$
\lambda(t)\cdot x = \bigoplus t^{m-n}M_x(a)^{mn}\,.
$$
Since the limit $\lim_{t\rightarrow 0} \lambda(t)\cdot x$ exist, we have 
$M_x(a)^{mn} = 0$ for all $m < n$. Let 
$$
M_x(w)^{\geq n} := \bigoplus_{m\geq n} M_x(w)^m\,.
$$
Then, $M(a)$ gives a map $M_x(s(a))^{\geq n}\ra M_x(t(a))^{\geq n}$ for all $n$. 
These subspaces determine subrepresentations $M_{x_n}$ of $M_x$ for all $n$. 
Since $y\in \rsf$, it follows that the corresponding filtration
$$
\cdots \supseteq M_n \supseteq M_{n+1} \supseteq \cdots
$$
where $M_n = M_x$ for $n \ll 0$ and $M_n = 0$ for $n \gg 0$, is strict filtration 
of $M_x$ in $\catfilt$. We also have 
$$
M_y \cong \bigoplus_{n\in \Z} M_n/M_{n+1}\,. 
$$

Conversely, suppose we have an strict filtration 
$$
\cdots \supseteq M_n \supseteq M_{n+1} \supseteq \cdots
$$
of $M_x$. This will determine a one-parameter subgroup $\lambda$ 
(by reversing the proof of first part or using the weight space decomposition at 
each vertex) such that the $\lim_{t\rightarrow 0} \lambda(t)\cdot x = y$ exist 
and $M_y \cong \bigoplus_{n\in \Z} M_n/M_{n+1} \in \catfilt$.
\end{proof}

Now onwards, we fix $\Theta\in \Gamma$ as in the Proposition \ref{abelian-condition}. 
Using Hilbert-Mumford criterion, we have the following:

\begin{proposition}\label{prop:semistab}
A point $x\in \rsf$ is $\chi_{\theta}$-semistable (respectively, $\chi_{\theta}$-stable) 
if and only if the corresponding object $M_x\in \catfilt$ is 
$\mu_{\Theta, \mathrm{rk}}$-semistable (respectively, $\mu_{\Theta, \mathrm{rk}}$-stable).
\end{proposition}

\begin{proposition}\label{prop-closed-orbits}
Let $x\in \rsf^{\mathrm{ss}}$. Then, the orbit of $x$ in $\rsf^{\mathrm{ss}}$ is 
closed if and only if the corresponding filtered representation $M_x$ is isomorphic 
to $\mathrm{gr^{max}}(M)$ in $\catfilt$.
\end{proposition}
\begin{proof}
Suppose that the orbit of $x$ in $\rsf^{\mathrm{ss}}$ is closed. If the corresponding 
filtered representation $M_x$ does not admit any strict filtration in $\catfilt$, then 
we are done. If there exists a strict filtration of $M_x$, then it also admits a maximal 
flag. By Proposition \ref{prop-1-ps-filtration} and \cite[Proposition 2.6]{Ki94}, we can 
conclude that $M_x$ is isomorphic to the associated graded of a maximal flag of $M_x$. 
This proves that $M_x$ is isomorphic to $\mathrm{gr^{max}}(M)$. The converse is immediate 
from Proposition \ref{prop-1-ps-filtration} and \cite[Proposition 2.6]{Ki94}.
\end{proof}

\begin{proposition} \label{prop:S-eq}
Two points $x, y\in \rsf^{\mathrm{ss}}$ are GIT equivalent, i.e., 
$$
\overline{Gx}\cap \overline{Gy}\cap \rsf^{\mathrm{ss}} \neq \emptyset
$$
if and only the corresponding
objects $M_x$ and $M_y$ are $S$-equivalent in $\catfilt$.
\end{proposition}
\begin{proof}
By \cite[Proposition 2.6]{Ki94}, if $x$ and $y$ are GIT equivalent, then there exist 
one-parameter subgroups $\lambda_1$ and $\lambda_2$ such that the integral pairing 
$\langle \chi_{\theta}, \lambda_1 \rangle = \langle \chi_{\theta}, \lambda_2 \rangle = 0 $ 
and $\lim_{t\ra 0} \lambda_1(t)\cdot x$ and $\lim_{t\ra 0} \lambda_2(t)\cdot y$ are in 
the same closed $G$-orbit in $\rsf^{\mathrm{ss}}$. Since $\lim_{t\ra 0} \lambda_1(t)\cdot x$ 
corresponds to the associated graded of a filtration in $\catfilt$, it follows that $M_x$ 
and $M_y$ have isomorphic maximal filtrations and hence they are $S$-equivalent. 

Conversely, if $M_x$ and $M_y$ are $S$-equivalent, then by Proposition 
\ref{prop-1-ps-filtration} there exist one-parameter subgroups $\lambda_1$ and 
$\lambda_2$ such that $\lim_{t\ra 0} \lambda_1(t)\cdot x$ and 
$\lim_{t\ra 0} \lambda_2(t)\cdot y$ are in the same closed $G$-orbit in $\rsf^{\mathrm{ss}}$.
\end{proof}

The proof of the following Proposition is inspired from \cite[Theorem 4.5]{AK07}.

\begin{proposition} \label{prop:localiso}
There is a local isomorphism $\tilde{h}_\mathrm{fil} \colon 
\fp{\rsf^{\mathrm{ss}}}/\fp{G} \ra \mff(\dv)$.
\end{proposition}
\begin{proof}
Let $f\colon S\ra \rsf^{\mathrm{ss}}$ be a morphism of $\mathbb{K}$-schemes, and let $\mathbb{M}_\mathrm{fil}$ be 
the restriction of the tautological family to $\rsf^{\mathrm{ss}}$. We define
$$
h_\mathrm{fil}\colon \fp{\rsf^{\mathrm{ss}}}\ra \mff(\dv)\; ~\quad \mbox{by}\;
(f\colon S\ra \rsf^{\mathrm{ss}})\mapsto [f^* \mathbb{M}_\mathrm{fil}]
$$

We have the following commutative diagram
\[
\xymatrix{
\fp{\rsf^{\mathrm{ss}}} \ar[r]^\iota \ar[d]_{h_\mathrm{fil}} & \fp{\rsc^{\mathrm{ss}}} 
\ar[d]^{h_\mathrm{rel}} \ar[dr]^\pi & \\
\mff(\dv) \ar[r]_j & \mfc(\dv) \ar[r]_{\tilde{h}_\mathrm{rel}} & \fp{\rsc^{\mathrm{ss}}}/\fp{G}
}
\]
For any $\mathbb{K}$-scheme $S$, the map $\fp{\rsf^{\mathrm{ss}}}(S) \ra
 \mff(\dv)(S)\times_{\mfc(\dv)(S)} \fp{\rsc^{\mathrm{ss}}}(S)$
defined by
$$
(f\colon S\ra \rsf^{\mathrm{ss}}) \mapsto (h_\mathrm{fil}^S(f), \iota \circ f)
$$
is a bijection. Hence, the square in the above diagram is Cartesian. This will induces 
the following Cartesian diagram
\[
\xymatrix{
\fp{\rsf^{\mathrm{ss}}}/\fp{G} \ar[r] \ar[d]_{\tilde{h}_\mathrm{fil}} & 
\fp{\rsc^{\mathrm{ss}}}/\fp{G} 
\ar[d]^{\tilde{h}_\mathrm{rel}} \\
\mff(\dv) \ar[r]_j & \mfc(\dv) 
} 
\] 
This proves that $\tilde{h}_\mathrm{fil} \colon 
\fp{\rsf^{\mathrm{ss}}}/\fp{G} \ra \mff(\dv)$
is a local isomorphism.
\end{proof}

Now, we can state and prove our main result. 

\begin{theorem}\label{main-theorem}
There exist a quasi-projective variety $\msf(\dimv{d})$ (respectively, $\mstf(\dimv{d}))$
which co-represents the moduli functor $\mff$ 
(respectively, $\mathfrak{M}_{\mathrm{fil}}^\mathrm{s}$).
\begin{enumerate}
\item The closed points of the moduli space $\msf(\dimv{d})$ (respectively, $\mstf(\dimv{d})$) 
correspond to the $S$-equivalence classes of $\mu_{\Theta, \mathrm{rk}}$-semistable 
(respectively, $\mu_{\Theta, \mathrm{rk}}$-stable) objects of $\catfilt (\mu)$ which have 
dimension vector $\dimv{d}$. 
\item There exists a canonical morphism $\phi \colon \msf(\dv)\ra \msc(\dv)$.
\end{enumerate}
\end{theorem}
\begin{proof} 
The moduli space $\msf(\dimv{d})$ is defined as the GIT quotient of $\rsf$ by $G$, hence 
it is a good quotient of $\rsf^{\mathrm{ss}}$ by $G$. In other words, the natural 
transformation  
$$
\fp{\rsf^{\mathrm{ss}}}/\fp{G} \ra \fp{\msf(\dimv{d})}
$$
corepresents the quotient functor $\fp{\rsf^{\mathrm{ss}}}/\fp{G}$. 
By \cite[Lemma 4.7]{AK07}, there exists a unique natural transformation 
$\eta\colon \mff(\dv)\ra \fp{\msf(\dimv{d})}$ such that the following diagram
\[
\xymatrix@!=3.5pc{
\fp{\rsf^{\mathrm{ss}}}/\fp{G} \ar[d]_{\tilde{h_\mathrm{fil}}} 
\ar[rd]^{\fp{\pi_{\mathrm{fil}}}} &\\
\mff(\dv) \ar[r]_{\eta} & \fp{\msf(\dimv{d})}
}
\]
commutes. This shows that $\msf(\dv)$ corepresents the moduli functor $\mff(\dv)$. 
By Proposition \ref{prop:S-eq}, it follows that the closed points of the moduli space 
$\msf(\dimv{d})$ correspond to the $S$-equivalence classes of objects of the category 
$\catfilt$. 

By the universal property of categorical quotients, we get a canonical morphism 
$$\phi \colon \msf(\dv)\ra \msc(\dv)$$
such that the following diagram
\[
\xymatrix{
\rsf \ar[r] \ar[d]_{\pi_{\mathrm{fil}}} & \rsc \ar[d]^{\pi_{\mathrm{rel}}} \\
\msf(\dv) \ar[r]_\phi & \msc(\dv)
}
\]
commutes.
\end{proof}

It should be notice here that there are two types in the description of closed points 
of $\msf(\dimv{d})$, which are as follows:
\begin{enumerate}
\item[(a)] A closed point of $\msf(\dimv{d})$ is said to be of \textbf{Type-1} if 
it correspond to the $S^\mathrm{JH}$-equivalence class of an object in $\catfilt(\mu)$.
\item[(b)] A closed point of $\msf(\dimv{d})$ is said to be of \textbf{Type-2} if 
it correspond to the $S$-equivalence class of an object in $\catfilt(\mu)$ which is 
not $S^\mathrm{JH}$-equivalence class.
\end{enumerate}

Let $R^{\mathrm{JH}}_\mathrm{fil}$ be a subscheme of $\rsf^\mathrm{ss}$ whose closed 
points are consisting of all those points $x$ for which the corresponding filtered 
representations have strict Jordan-H\"older filtrations.

\begin{proposition}
The scheme $R^{\mathrm{JH}}_\mathrm{fil}$ is a $G$-invariant  open subscheme of 
$\rsf^\mathrm{ss}$, and it admits a good quotient
$$
\pi_{\mathrm{JH}}\colon R^{\mathrm{JH}}_\mathrm{fil}\ra M_{\mathrm{fil}}^\mathrm{JH}(\dv)
$$ 
such that the following diagram is Cartesian:
\[
\xymatrix{
R^{\mathrm{JH}}_\mathrm{fil} \ar[r] \ar[d]_{\pi_{\mathrm{JH}}} & \rsc^\mathrm{ss} 
\ar[d]^{\pi_\mathrm{rel}} \\
M_{\mathrm{fil}}^\mathrm{JH}(\dv) \ar[r]_\eta & \msc(\dv)
}
\]
where $\eta$ is an open immersion.
\end{proposition}
\begin{proof}
Since $\rsf^\mathrm{ss}$ is $G$-invariant open subset of $\rsc^\mathrm{ss}$, 
it follows that
$$
M_{\mathrm{fil}}^\mathrm{JH}(\dv):= 
\msc(\dv) - \pi_\mathrm{rel}(\rsc^\mathrm{ss} - \rsf^\mathrm{ss}) 
$$
is an open subset of $\msc(\dv)$. Then, it is easy to see that
$R^{\mathrm{JH}}_\mathrm{fil} = \pi_\mathrm{rel}^{-1}(M_{\mathrm{fil}}^\mathrm{JH}(\dv))$ 
is a $G$-invariant open subscheme of $\rsf^\mathrm{ss}$. Hence, we get a good quotient 
$\pi_{\mathrm{JH}}\colon R^{\mathrm{JH}}_\mathrm{fil}\ra M_{\mathrm{fil}}^\mathrm{JH}(\dv)$
which is the restriction of $\pi_{\mathrm{rel}}$ to an open subscheme $R^{\mathrm{JH}}_\mathrm{fil}$.
\end{proof}

\begin{proposition}\label{prop-JH-immersion}
There exist an open immersion $\sigma \colon M_{\mathrm{fil}}^\mathrm{JH}(\dv)\ra \msf(\dv)$ such that
the following diagram
\[
\xymatrix{
R^{\mathrm{JH}}_\mathrm{fil} \ar[r]^i \ar[d]_{\pi_{\mathrm{JH}}} & \rsf^\mathrm{ss}
\ar[d]_{\pi_\mathrm{fil}} \ar[r]^j & \rsc^\mathrm{ss} \ar[d]_{\pi_\mathrm{rel}} \\
M_{\mathrm{fil}}^\mathrm{JH}(\dv)  \ar[r]_\sigma \ar@/_2pc/[rr]_\eta & \msf(\dv) 
\ar[r]_\phi & \msc(\dv) 
}
\]
is commutative;
where $i$ and $j$ are open embeddings, $\phi$ is a canonical morphism as in the 
Theorem \ref{main-theorem}, and all three vertical morphisms are good quotients.
\end{proposition}
\begin{proof}
The existence of $\sigma$ and $\eta = \phi \circ \sigma$ follows from the universal properties. The proof of open immersion is
similar to the case of $\eta$ as above.
\end{proof}
\begin{cor}\label{cor-open-immersion}
If $\mathrm{rk}$ is zero on torsion classes, then the canonical morphism 
$$\phi \colon \msf(\dv)\ra \msc(\dv)$$
is an open immersion. In this case, the closed points of $\msf(\dv)$ are precisely of 
\textbf{Type-1} which correspond to $S^{\mathrm{JH}}$-equivalence classes of objects 
of $\catfilt(\mu)$.
\end{cor}
\begin{proof}
If $\mathrm{rk}$ is zero on torsion classes, then we have 
$R^{\mathrm{JH}}_\mathrm{fil} = \rsf^\mathrm{ss}$ (cf. Example \ref{ex:sink slope}).
The result now follows from Theorem \ref{main-theorem} and Proposition 
\ref{prop-JH-immersion}.
\end{proof}

\begin{remark}\leavevmode \rm{
\begin{enumerate}
\item If we consider the trivial quiver in place of $Q$, then the filtered moduli 
space $\msf(\dv) = \msc(\dv)$ is just a point for any increasing sequence of dimension 
vector $d_1 \leq d_2 \leq \ldots \leq d_\ell$. The filtered moduli space  $\msf(\dv)$ 
is empty for (non-increasing) other dimension vectors. This follows from the 
semi-simplicity of representations of $A_\ell$ quivers.
\item If $\ell = 1$, then we are in the situation of \cite{Ki94} and hence for any 
$\Theta\in \Gamma$, we get projective moduli spaces.

\item In general, the moduli spaces $\msf(\dv)$ may not be projective varieties.  
For certain choice of stability parameter $\Theta$ and dimension vector $\dv$, we 
may get projective moduli space of filtered representations of $Q$. In particular, 
if the open set $\rsf$ is same as $\rsc^\mathrm{ss}$, then the GIT quotient $\msf(\dv)$ 
is projective.

For example,  let us consider the quiver $S :=A_2 \times A_2$, i.e.,
\begin{equation*}
\quad \xymatrix@C=3.0pc{
v_1 \ar[r]^{\beta_1} \ar[d]_{\alpha_1} & v_2 \ar[d]^{\alpha_2} \\
w_1 \ar[r]_{\beta_2} & w_2
}
\end{equation*}
where the set of vertices $S_0 = \{v_1, v_2, w_1, w_2\}$.
Let us take $\dv = (1, 1, 1, 2)$ and $\Theta = (1, 0, 1, -1)$. With direct computations,
one can check that the representation $M$ of $S$ having dimension vector $\dv$ satisfying 
the relation $\alpha_2\beta_1 - \beta_2\alpha_1$ is $\Theta$-semistable if and only if 
$M(\beta_1)$ and $M(\beta_2)$ are injective maps. Hence, in this case, we have 
$\rsf = \rsc^\mathrm{ss}$.
\end{enumerate}

}
\end{remark}
\section{Determinantal theta functions}\label{sec-det-fns}
In this section, we assume that the characteristic of the field $\mathbb{K}$ is zero.
We will use results of \cite{DW1, DW2} to give an explicit embedding of moduli of 
filtered quiver representations in a projective space. 
First, we recall the definition of semi-invariants $\bar{c}^{N'}$ for quiver with 
relations.
Let $N'$ be an object in $\catchain$. Let 
$$
\tilde{P_1}\ra \tilde{P_0} \ra N' \ra 0
$$
be the minimal presentation of $N'$ in $\catchain$, see \cite[Section 2]{DW2} for 
more details.

For any object $M'$ in $\catchain$, the semi-invariant $\bar{c}^{N'}$ is defined by 
$\bar{c}^{N'}(M')$ to be the determinant of the matrix
$$
\Hom[\catchain]{\tilde{P_0}}{M'}\ra \Hom[\catchain]{\tilde{P_1}}{M'}
$$
whenever it is a square matrix.

Let $\rsc^1, \dots, \rsc^m$ be the irreducible components of $\rsc$. The semi-invariants
$\bar{c}^{N'}$ is defined on the components $\rsc^j$ of $\rsc$ for which 
$$
\dim \Hom[\catchain]{\tilde{P_0}}{M'} = \dim \Hom[\catchain]{\tilde{P_0}}{M'}
$$
while $\dim \Hom[\catchain]{N'}{M'} = 0$ for general $M'\in \rsc^j$.

In view of \cite[Theorem 1]{DW2}, we may work with the faithful components of $\rsc$. 
Recall that a component $\rsc^j$ of $\rsc$ is called \emph{faithful} if every element 
of $\mathbb{K}(\QAl)$ acting trivially on every module from $\rsc^j$ is in the admissible 
ideal $I$. Let $\rsc^\mathrm{ff}$ be the union of faithful components of $\rsc$.

\begin{proposition}\label{det-prop-1}
A representation $M_x$ corresponding to a point $x\in \rsc^\mathrm{ff}$ 
is $\mu_{\Theta, \mathrm{rk}}$-semistable if and only if there is an object 
$N'\in \catchain$ of projective dimension one with the minimal projective resolution 
\begin{equation}\label{eq-det-1}
0 \ra \tilde{P}_1 \stackrel{\gamma'}{\ra} \tilde{P}_0 \ra N'\ra 0,
\end{equation}
where $\tilde{P}_0 =  \bigoplus_{w\in (\QAl)_0} U_{0_w}\otimes P'_w$ and 
$\tilde{P}_1 =  \bigoplus_{w\in (\QAl)_0} U_{1_w}\otimes P'_w$ are projective modules 
in $\catchain$, such that the following holds:
\begin{enumerate}
\item\label{c1} For each $w\in (\QAl)_0$, we have 
$\dim U_{0_w} - \dim U_{1_w} = n(\Theta(\dv)r_w - \mathrm{rk}(\dv) \Theta_w)$ 
for some positive integer $n$.
\item The map
$$
\Hom[]{\gamma'}{M_x}\colon \Hom[]{P'_0}{M_x}\ra \Hom[]{P'_1}{M_x}
$$ is invertible, i.e. $\Theta_{\gamma'}(M_x):= \det \Hom[]{\gamma'}{M_x} \neq 0$.
\end{enumerate}
\end{proposition}
\begin{proof}
Assume that $M_x$ is $\mu_{\Theta, \mathrm{rk}}$-semistable, i.e., the point 
$x$ is $\chi_{\theta}$-semistable, and hence there exist a semi-invariant $f$ with 
weight $\chi_{\theta}^n$ such that $f(x) \neq 0$. 
By \cite[Theorem 1]{DW1}, there exists a representation $N$ of $\QAl$ with the 
minimal projective resolution 
$$
P_1 \stackrel{\gamma}{\ra} P_0 \ra N\ra 0
$$
having the Euler form $\langle \dim M_x, \dim N\rangle = 0$ such that 
$c^{N}(M_x) \neq 0$. From the proof of \cite[Proposition 1]{DW2}, we have a minimal
projective presentation
$$
\tilde{P_1} \stackrel{\gamma'}{\ra} \tilde{P_0} \ra N'\ra 0,
$$
where $N' := N/IN$, and a determinant semi-invariant $\bar{c}^{N'}$ is the restriction 
of $c^{N}$ on the faithful component. By \cite[Theorem 1]{DW2}, the 
$\mathbb{K}(\QAl)/I$-module $N'$ is of projective dimension one. 

Since the determinant semi-invariant $c^N$ has weight $\chi_{\theta}^n$, it follows 
from \eqref{eq-chi-theta} that
$$
\dim U_{0_w} - \dim U_{1_w} = n(\Theta(\dv)r_w - \mathrm{rk}(\dv) \Theta_w)
$$
for each vertex $w\in (\QAl)_0$.

Conversely, assume that there is an object $N'\in \catchain$ of projective dimension 
one with the projective presentation \eqref{eq-det-1} satisfying the given conditions. 
Then, $\Theta_{\tilde{\gamma}}$ can be viewed as a semi-invariant of weight 
$\chi_{\theta}^n$. To see this, for given multiplicity vector spaces 
$\U= \{U_{0_w}, U_{1_w}\}_{w\in (\QAl)_0}$ satisfying (\ref{c1}), one can consider 
the associated character 
\begin{equation}\label{chi-U-defn}
\chi_U\big((g_w)_{w\in (\QAl)_0}\big) := 
\prod_{w\in (\QAl)_0} \det(g_w)^{(\dim U_{0_w} - \dim U_{1_w})} \,.
\end{equation}
This proves that $M_x$ is $\mu_{\Theta, \mathrm{rk}}$-semistable.
\end{proof}

\begin{remark}\rm{
In the above result, if we drop the condition on vector spaces $\U$ as in (\ref{c1}), 
then what we can conclude is that the point $x$ is $\chi_U$-semistable which, in general, 
may not correspond to the $\mu_{\Theta, \mathrm{rk}}$-semistability (see \eqref{eq-chi-theta}).
}
\end{remark}

\begin{cor}\label{det-cor-1}
A representation $M_x$ corresponding to a point $x\in \rsc^\mathrm{ff}\cap \rsf$ is 
$\mu_{\Theta, \mathrm{rk}}$-semistable if and only if  either 
there is an object $N''\in \catfilt$ 
of projective dimension one with a minimal projective resolution
\begin{equation}\label{eq-det-2}
0 \ra \tilde{P_1} \stackrel{\tilde{\gamma}}{\ra} \tilde{P_0} \ra N''\ra 0,
\end{equation}
\rm{or} there exists an epi-monic $\tilde{P_1} \stackrel{\tilde{\gamma}}{\ra} \tilde{P_0}$ 
in $\catfilt$; where $\tilde{P_0}$ and $\tilde{P_1}$ are projective objects in $\catfilt$ 
as in Proposition \ref{det-prop-1} (see Proposition \ref{filtered projective} for 
justification of this definition), such that the map
$$
\Hom[]{\tilde{\gamma}}{M_x}\colon \Hom[]{\tilde{P_0}}{M_x}\ra \Hom[]{\tilde{P_1}}{M_x}
$$ is invertible, i.e. $\Theta_{\tilde{\gamma}}(M_x):= \det \Hom[]{\tilde{\gamma}}{M_x} \neq 0$.
\end{cor}
\begin{proof}
Recall that the inclusion functor $\iota \colon \catfilt \ra \catchain$ has a left 
adjoint
$$
\kappa \colon \catchain  \ra \catfilt
$$
(see \cite{SS16}) and being left adjoint $\kappa$ is a right exact functor. Since 
$\kappa$ has an exact right adjoint functor $\iota$, it carries projective objects 
to projective objects. In fact, the indecomposable projective modules $P'_w$ are 
projective objects in $\catfilt$ (see Proposition \ref{filtered projective}). 

Hence, if $\kappa(N')$ is a non-zero object in $\catfilt$, then the functor $\kappa$ 
takes any presentation of the form \ref{eq-det-1} to a presentation of the form 
\ref{eq-det-2}, where $N'' = \kappa(N')$.
i.e., $\tilde{\gamma} = \kappa(\gamma')$, for $\gamma'$ as in the Proposition 
\ref{det-prop-1}. In fact, we have the identification
$$
\Hom[\catfilt]{\kappa(\gamma')}{M} \cong \Hom[\catchain]{\gamma'}{\iota(M)}
$$
using the adjunction between $\iota$ and $\kappa$.
If $\kappa(N') = 0$, then we get an epi-monic 
$\tilde{P_1} \stackrel{\tilde{\gamma}}{\ra} \tilde{P_0}$ satisfying the 
required conditions.
\end{proof}

\begin{proposition}
There exists an ample line bundle $\lambda_\U(\dv)$ on the moduli space $\msc(\dv)$.
\end{proposition}
\begin{proof} 
Let $M$ be a flat family of objects of $\catchain$ having dimension vector $\dv$ over 
a scheme $S$. 
For non-zero vector spaces $\U= \{U_{0_w}, U_{1_w}\}_{w\in (\QAl)_0}$ satisfying 
\begin{equation}\label{eq-U-condn}
\dim U_{0_w} - \dim U_{1_w} = (\Theta(\dv) r_w - \mathrm{rk} (\dv) \Theta_w)\,,
\end{equation}
we define a line bundle
$$
\lambda_\U(M) := \Big(\det \bigoplus_w 
\Hom[]{U_{0_w}}{M(w)} \Big)^{-1}\bigotimes \Big(\det \bigoplus_w 
\Hom[]{U_{1_w}}{M(w)} \Big)
$$
on $S$.
From \eqref{eq-U-condn}, we have
\begin{equation}\label{eq-theta-1}
\sum_{w\in (\QAl)_0}  (\dim U_{0_w} - \dim U_{1_w})d_w = 0.
\end{equation}

For any map
$\gamma' \colon P'_1 \ra P'_0\,,$
where $$P'_0 = \displaystyle \bigoplus_{w\in (\QAl)_0} U_{0_w}\otimes P'_w ~\mbox{and}~  
P'_1 = \displaystyle \bigoplus_{w\in (\QAl)_0} U_{1_w}\otimes P'_w\,,$$
we get a (global) section $\Theta_{\gamma'}(M):= \det \Hom[]{\gamma'}{M}$ of 
$\lambda_\U(M)$ on $S$.

If $M$ is a flat family of $\mu_{\Theta, \mathrm{rk}}$-semistable objects of 
$\catchain$ having dimension vector $\dv$ over a scheme $S$, then we get a formal 
line bundle on the moduli functor $\mfc(\dv)$.

Let $\mathbb{M}$ be a tautological family of $\mu_{\Theta, \mathrm{rk}}$-semistable 
objects of $\catchain$ on $\rsc^\mathrm{ss}$. Then by \eqref{eq-theta-1}, it follows 
that $\lambda_\U(\mathbb{M})$ is a $G$-linearized line bundle.

To conclude, using Kempf's descent criterion, that the $G$-linearized line bundle
$\lambda_\U(\mathbb{M})$ descends to a line bundle on the moduli space $\msc(\dv)$, 
we need to verify that for each point $x\in \rsc^\mathrm{ss}$ in closed orbit, the 
isotropy group of $x$ acts trivially on the fibre over $x$. To see this, let 
$x\in \rsc^\mathrm{ss}$ is in closed orbit.
Then, the corresponding module $M_x$ in $\catchain$ is 
$\mu_{\Theta, \mathrm{rk}}$-polystable, i.e.,
$$
M_x \cong \bigoplus_{k=1}^m L_k \otimes M_k\,,
$$
where $M_k$ are non-isomorphic $\mu_{\Theta, \mathrm{rk}}$-stable objects of dimension 
vector $\dv_k$ with the slope 
\begin{equation}\label{eq-slope-equal}
\frac{\Theta(\dv_k)}{\mathrm{rk}(\dv_k)} = \mu_{\Theta, \mathrm{rk}}(M_k) = 
\mu_{\Theta, \mathrm{rk}}(M)
= \frac{\Theta(\dv)}{\mathrm{rk}(\dv)}
\end{equation}
and $L_k$ are multiplicity vector spaces.
Note that
$$
\lambda_\U(M_x) \cong \bigotimes_{k=1}^m (\det L_k)^{u_k} \otimes \lambda_\U(M_k)^{\dim L_k}\,,
$$
where $$u_k = \sum_{w\in (\QAl)_0}  (\dim U_{0_w} - \dim U_{1_w})d_{k_v}.$$

Note that
\[
\begin{array}{ll}
\sum_{w\in (\QAl)_0}  (\dim U_{0_w} - \dim U_{1_w})d_{k_w}
& = \sum_{w\in (\QAl)_0} \big(\Theta(\dv)r_w - \mathrm{rk}(\dv) \Theta_w\big)d_{k_w} ~\quad \quad
(\mathrm{by}~~ \eqref{eq-U-condn} )\\
& \\
& = \Theta(\dv)\mathrm{rk}(\dv_k) - \Theta(\dv_k)\mathrm{rk}(\dv)\\
& \\
& = 0 ~\quad \quad
(\mathrm{by}~~ \eqref{eq-slope-equal} )\\
\end{array}
\]
Since $M_k$ are $\mu_{\Theta, \mathrm{rk}}$-stable, the isotropy group of $x$ of $G$ 
is isomorphic to $\prod_{k=1}^m \mathrm{GL}(L_k)$. Hence, from the above computations, 
it follows that the isotropy group acts trivially on the fibre over $x$.

Therefore, by Kempf's descent criterion, we get a line bundle $\lambda_\U(\dv)$ on the 
moduli space $\msc(\dv)$. 
Moreover, a global section $\Theta_{\gamma'}(\mathbb{M})$ being a $G$-invariant section 
of $\lambda_\U(\mathbb{M})$, it descends to a section $\Theta_{\gamma'}(\dv)$ of 
$\lambda_\U(\dv)$ (cf. \cite[Proposition 7.5]{AK07}).
Since $\lambda_\U(\mathbb{M})$ is a $G$-linearized line bundle which is used to 
construct the moduli space $\msc(\dv)$, it follows that the descended line bundle 
$\lambda_\U(\dv)$ is ample on $\msc(\dv)$. Moreover, we have
$$
H^0(\rsc, \lambda_\U(\mathbb{M}))^G \cong H^0(\msc(\dv), \lambda_\U(\dv))
$$
(cf. \cite[Proposition 7.6]{AK07}).
\end{proof}

Using \cite[Theorem 1]{DW1}, for sufficiently large dimensional vector spaces $\U$ 
satisfying \eqref{eq-theta-1}, we get finitely many maps $\gamma_0, \dots, \gamma_m$ 
as in the proof of Proposition \ref{det-prop-1} such that the corresponding map
$$
\Theta_{\gamma} \colon \ms(\dv) \ra \mathbb{P}^m
$$
is a closed scheme-theoretic embedding. 

Let $\msc(\dv)^\mathrm{ff}$ be the closed subscheme of $\msc(\dv)$ corresponding 
to $\rsc^\mathrm{ff}$. Then, we have the following result which follows from the 
similar arguments as in \cite[Theorem 7.8]{AK07}.

\begin{proposition}\label{det-prop-2}
For any dimension vector $\dv$, we can find finite-dimensional vector spaces 
$\U = \{U_{0_w}, U_{1_w}\}_{w\in (\QAl)_0}$ satisfying \eqref{eq-theta-1}
and finitely many maps $\gamma'_0, \dots, \gamma'_m$ as in Proposition 
\ref{det-prop-1} such that the map
$$
\Theta_{\gamma'} \colon \msc(\dv)^\mathrm{ff} \ra \mathbb{P}^m
$$
given by $[M]\mapsto (\Theta_{\gamma'_0}(M) : \cdots : \Theta_{\gamma'_N}(M))$ is 
scheme-theoretic closed embedding.
\end{proposition}

We notice that this embedding $\Theta_{\gamma'}$ is precisely the restriction of 
$\Theta_{\gamma}$ (cf. Proposition \ref{det-prop-1}).
Let $M_{\mathrm{fil}}^\mathrm{JH}(\dv)^\mathrm{ff} := 
M_{\mathrm{fil}}^\mathrm{JH}(\dv) \cap \msc(\dv)^\mathrm{ff}$ 
(respectively, $\msf(\dv)^\mathrm{ff} := \msf(\dv) \cap \msc(\dv)^\mathrm{ff}$) . Then $M_{\mathrm{fil}}^\mathrm{JH}(\dv)^\mathrm{ff}$ and $\msf(\dv)^\mathrm{ff}$ are 
locally closed subschemes of  $\msc(\dv)$. 

\begin{cor}\label{det-cor-2}
For any dimension vector $\dv$, we can find finite dimensional vector spaces 
$\U = \{U_{0_w}, U_{1_w}\}_{w\in (\QAl)_0}$ 
satisfying \eqref{eq-theta-1} and finitely many maps 
$\tilde{\gamma_0}, \dots, \tilde{\gamma_m}$ as in 
Corollary \ref{det-cor-1} such that the map
$$
\Theta_{\tilde{\gamma}} \colon M_{\mathrm{fil}}^\mathrm{JH}(\dv)^\mathrm{ff} \ra \mathbb{P}^m
$$
given by 
$[M]\mapsto (\Theta_{\tilde{\gamma_0}}(M) : \cdots : \Theta_{\tilde{\gamma_m}}(M))$ 
is a scheme-theoretic locally closed embedding.
\end{cor}
\begin{proof}
There is an ample line bundle $\tilde{\lambda}_U(\dv) = \eta^*\lambda_U(\dv)$ on the 
moduli space $M_{\mathrm{fil}}^\mathrm{JH}(\dv)$, and for $\tilde{\gamma_i} = \kappa(\gamma_i)$, 
we have 
$$
\eta^*\Theta_{\gamma_i} = \Theta_{\tilde{\gamma_i}}
$$
Now using Corollary \ref{det-cor-1}, the result follows by combining the embedding 
$\eta \colon M_{\mathrm{fil}}^\mathrm{JH}(\dv)^\mathrm{ff} \ra \msc(\dv)^\mathrm{ff}$ 
and $$\Theta_\gamma \colon \msc(\dv)^\mathrm{ff} \ra \mathbb{P}^m \, .$$ 
\end{proof}

In particular if we choose suitable slope function we get following result.

\begin{cor}\label{det-cor-3}
For any dimension vector $\dv$ and a slope function $\mu_s$ as in Example 
\ref{ex:sink slope}, we can find finite dimensional vector spaces 
$\U = \{U_{0_w}, U_{1_w}\}_{w\in (\QAl)_0}$ satisfying \eqref{eq-theta-1} and finitely 
many maps $\tilde{\gamma_0}, \dots, \tilde{\gamma_m}$ as in Corollary \ref{det-cor-1} 
such that the map
$$
\Theta_{\tilde{\gamma}} \colon \msf(\dv)^\mathrm{ff} \ra \mathbb{P}^m
$$
given by $[M]\mapsto (\Theta_{\tilde{\gamma_0}}(M) : \cdots : \Theta_{\tilde{\gamma_m}}(M))$ 
is a scheme-theoretic locally closed embedding.
\end{cor}
\begin{remark}\rm{
For the components which are not faithful, we can have the similar statements as in 
the Proposition \ref{det-prop-2} and the Corollary \ref{det-cor-2}, \ref{det-cor-3}
by further introducing extra relations. More precisely, the projective modules in 
presentations \eqref{eq-det-1} and \eqref{eq-det-2} will change up to extra relations.
}
\end{remark}

\subsection*{Acknowledgement} The first-named author would like to acknowledge the 
support of SERB-DST (India) grant number YSS/2015/001182 and MTR/2018/000475. The 
second name author would like to acknowledge the support of DAE and DST INSPIRE grant number 
IFA 13-MA-25. We also thank HRI and IIT Gandhinagar for providing excellent infrastructure.
%
%
%
%

\end{document}